\documentclass[12pt]
{amsart}

\usepackage[all]{xy}
\usepackage{amsmath,amsthm,amssymb,graphicx,eucal, amscd, ifthen} 
\usepackage[margin=1in]{geometry} 
\usepackage{cite} 
\usepackage[colorlinks=true]{hyperref} 
\hypersetup{urlcolor=blue, citecolor=red}

\author{Benjamin Steinhurst} 
\address{Benjamin Steinhurst, Department of Mathematics and Computer Science, McDaniel College, Westminster, MD 21157}
 \email{bsteinhurst@mcdaniel.edu}
 
\author{Alexander Teplyaev}
\address{Alexander Teplyaev, Department of Mathematics, University of Connecticut, Storrs, CT 06269-1009}
     \email{teplyaev@uconn.edu}

\title[Spectral Analysis on Barlow-Evans Projective Limit Fractals]{Spectral Analysis on Barlow-Evans Projective Limit Fractals}

\thanks{Research supported in part by the National Science Foundation, grant DMS-1613025.}

\newcommand{\C}[1]{\ensuremath{\mathcal{#1}}}
\newcommand{\B}[1]{\ensuremath{\mathbb{#1}}}
\newcommand{\F}[1]{\ensuremath{\mathfrak{#1}}}
\newcommand{\pis}[1]{\ensuremath{\pi^*_{#1}}}
\newcommand{\phis}[1]{\ensuremath{\phi^*_{#1}}}

\def\mG#1{\mu_{G_{#1}}}
\def\mF#1{\mu_{F_{#1}}}

\newtheorem{definition}{Definition}[section]
\newtheorem{theorem}{Theorem}[section]
\newtheorem{corollary}{Corollary}[section]
\newtheorem{lemma}{Lemma}[section]
\newtheorem{proposition}{Proposition}[section]
\newtheorem{remark}{Remark}[section]

\begin{document}
\maketitle

% Enter the first author's name and address:

\begin{abstract}
We develop the foundation of the spectral analysis on Barlow-Evans projective limit fractals, or vermiculated spaces, which corresponds to symmetric Markov processes on these spaces. For some new examples, such as the generalized Laakso spaces and a Spierpinski P\^ate \`a Choux, one can develop a complete spectral theory, including the eigenfunction expansions that are analogous to Fourier series. Also, one can construct connected fractal spaces isospectral to the fractal strings of Lapidus and van Frankenhuijsen. Our work is motivated by recent progress in mathematical physics on fractals.  
 \tableofcontents
\end{abstract}

\section{Introduction}\label{sec:intro}
Analysis on projective, or inverse, limit spaces 
is an active area of current research \cite[and 
references 
therein]{Cheeger2013,Cheeger2015}. 
We study symmetric regular Dirichlet forms \cite{CF,FOT} on the fractal-like spaces $F_\infty$ constructed in \cite{BE04}. 
Our motivation primarily comes from applications in mathematical physics, see 
\cite[and 
references 
therein]{ABDTV12,Akk,ADT,ADT10,AR,PAR18,CTT}. In particular, \cite{PAR18} shows that explicit formulas for kernels of spectral operators, such as heat kernel and Schro\"odinger kernels,  can be obtained for these types of fractal spaces. 
The main results of our paper, Theorems  \ref{thm:specdecomp} and \ref{thm:specproj}, deal with the spectrum and the spectral resoltion of the Laplaican on the Barlow-Evans type projective limit space. 

Barlow and Evans in \cite{BE04}    used projective limits to produce a new class of state spaces for Markov processes. They also construct a projective limit Markov process by taking the projective limit of a sequence of compatible resolvent operators. We shall build, in a similar manner, a projective sequence of Dirichlet forms which we will then show have a non-degenerate limit. The projective sequences are built from a base Dirichlet space, that is a metric measure space equipped with a Dirichlet form together with its domain and a sequence of ``multiplier spaces.''  We will show that for reasonable base and index spaces one can develop a complete spectral theory of the associated Laplace operators, including formulas for spectral projections, utilizing the tools of Dirichlet form theory on the projective limit space, $F_{\infty}$. The characterization of the spectra of the Laplacians presented here is a generalization of those obtained previously by the first author for Laakso spaces in \cite{RS09,Steinhurst2010}. It is worth noting that the construction of $F_{\infty}$ in this paper is the same as that in \cite{BE04} and while the analytic apparatus is different (Dirichlet forms vs. resolvents) the constructions are in the same spirit. 

Given a measure space on which one has a Laplacian it is natural to study the spectrum. As the measure space becomes more complicated this task can become very difficult. On fractal spaces such as the Sierpinski gasket and carpet this problem has been extensively studied \cite{S,BB99,BBKT,K,K03}. For finitely ramified self-similar highly symmetric fractals a complete spectral analysis is  possible although rather complicated, see \cite{eigen1,eigen2} and references therein. %\footnote{Cite something for non-self-similar fractals. --5Aug2015BAS} 
Moreover, it is possible to extend this kind of spectral analysis to finitely ramified fractafolds, that is to metric measure spaces that have local charts from open sets of a reference fractal as opposed to $\mathbb{R}^{d}$. This is one way of obtaining new examples from old, including isospectral fractafolds, see \cite{S,Sf,IRS,RS-AT}. The projective limit construction provides yet another way of controllably obtaining new measure spaces and in this paper we examine how the spectral data transfers to the limit space from the base space.

The main goal of this paper is an understanding of the spectrum of a class of Laplacians. We have found it more straight forward to work in terms of the associated Dirichlet forms. This is particularly noticeable in Definition \ref{def:DF}, where the domain of a Dirichlet form is easier to describe than the domain of the corresponding Laplacian. 

We discuss in the final section of this paper how the projective limit construction can produce connected fractals which are isospectral to a given fractal string (see \cite{LvF} and references therein). 
This makes it possible to make a connection between Laplacians and spectra on fractal strings and on connected fractals in a natural way. 
%This makes possible a natural connection between fractal strings and the study of Laplacians on fractals. 
Determining heat kernel estimates for Laplacians on fractal spaces has a long tradition but actual analysis of heat kernels on specific fractals is beyond the scope of this paper  (for references most relevant to our work see   instance, \cite{ABKT,BV1,BV2,Ba03,Ba04,BCG,BGK,BK,Steinhurst2013,G}). For example Laakso spaces have Gaussian heat kernel estimates while Sierpinski gasket-like fractals have sub-Gaussian estimates often depending on geometric conditions.  

One note of caution, our analysis of fractals defined as projective limits is an entirely intrinsic analysis on abstractly defined objects. Even in the simplest examples, Laakso spaces, the limit space is not bi-Lipschitz embeddable in any finite dimensional Euclidean space, \cite{Laakso2000}. 
However Laakso spaces provide a useful set of examples for a general theory which attempts to reprove the main results of differential geometry on possibly fractal spaces with  regular Dirichlet forms, see \cite{T1,T2,T3,T4,T5}. 

We begin with definitions in Section \ref{sec:def}. In Sections \ref{sec:projlimits} and \ref{sec:projlap} we provide the background on projective systems of measure spaces along with the limiting procedure for the Laplacians on each approximating measure space. Section \ref{sec:main} contains the main results of the paper which give a decomposition of the spectrum of the Laplacian on the limit space. Then in Section \ref{sec:examples} we describe three classes of examples of spaces that can be constructed with this method.

\section{Definitions}\label{sec:def}
The following definitions are 
%based on those given in 
essentially repeated from \cite{BE04}.
% and references therein.

Let $F_0$ be a locally compact, second-countable, Hausdorff space with a $\sigma-$finite Borel measure $\mF0$. In addition we assume there is a sequence of compact, second-countable, Hausdorff spaces $G_i$ for $i \ge 1$ with Borel probability measures $\mG{i}$. The measures $\mF0$ and $\mG{i}$ are all assumed to be Radon measures with full support. 

We call $F_0$ the horizontal base space, and call $G_i$ the vertical multiplier spaces, see Figure \ref{fig:Laakso}.
%\footnote{I thought that this visual description may be helpful, but please check that the notions of horizontal and vertical are ok, or otherwise interchange or delete this sentence. -- We don't have any pictures in the paper at the moment where this is contradicted so we can keep it ready if we want to add pictures to the examples in Section 6. \\ This would be an extra confusion without a picture. I will add one or two in.}  

Inductively we define a sequence of locally compact topological measure spaces and maps between them as follows (refer to Figure~\ref{fig:sequence}.) Suppose that $F_{i-1}$ for $i \ge 1$ is defined as a locally compact, second-countable, Hausdorff space and $B_i \subset F_{i-1}$ is a closed subset. 

\begin{figure}[t]
\begin{center} \ 
	\xymatrix{ 
		& \ar@{.}[d] & \ar@{.}[d] & & \\
		% {\lim_{\leftarrow}F_i} \ar[dll]^{\Phi_2} \ar[ddll]^{\Phi_1} \ar[dddll]^{\Phi_0} \\
		&F_1\times G_2 \ar[dr]_{\psi_2} \ar[r]^{\pi_2} & F_2 \ar[d]^{\phi_2}&&\\
		&F_0\times G_1 \ar[dr]_{\psi_1} \ar[r]^{\pi_1} & F_1\ar[d]^{\phi_1}&&\\ 
		& & F_0&&\\
	}
	\end{center}
	\caption{The sequence of spaces and the maps between them in Definition \ref{def:Fi}.}\label{fig:sequence}
\end{figure}
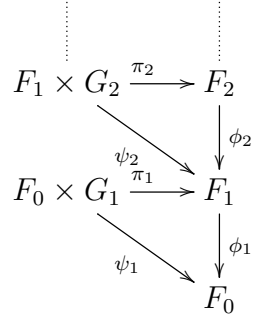

\begin{definition}\label{def:Fi}
Set
$$F_i = ((F_{i-1} \setminus B_i) \times G_i \bigcup B_i$$
and
$$\pi_i(x,g) = \left\{ \begin{array}{ll}(x,g) & if\ x \in F_{i-1} \setminus B_i \\ x & if\ x \in B_i. \end{array}\right. $$
The space $F_i$ is topologized by the map $\pi_i$, which means that a subset of $F_i$ is open if and only if its $\pi_i$-preimage is open in $F_{i-1}\times G_i$.
\end{definition}

The maps $\psi_i$ are the natural projections $F_{i-1} \times G_i \rightarrow F_{i-1}$ and define $\phi_i = \psi_i \circ \pi_i^{-1}\ : F_i \rightarrow F_{i-1}$. Alternatively $\phi_i$ can be defined by 
$$\left. \begin{array}{rl}\phi_i(x,g) =x & if\ \ x \in F_{i-1} \setminus B_i \\ \phi_i(x) =x & if\ \ x \in B_i. \end{array}\right.$$

\begin{definition}\label{def:muFi}Given $\mF0$  
We inductively define measures $\mF{i}$ on $F_i$ for $i \ge 1$ by
$$\mF{i}(\cdot) := (\mF{i-1} \times \mG{i})( \pi_i^{-1}(\cdot)).$$
The measure $\mF{i}$ is defined on the Borel $\sigma-$algebra generated by the above defined topology on $F_i$.
\end{definition}

%\begin{definition}\label{def:BEsequence}
The sequence of spaces and associated maps $\{F_i,G_i,\phi_i,\pi_i,\psi_i\}$ will be called a Barlow-Evans sequence. We also assume that the sequence of measures  $\mF{i}$ defined above is fixed. 
%\end{definition}
Note that $\phi_i$ is an open map because $\psi_i$ is open by virtue of it being a projection.

Since $F_i$ is locally compact, second-countable, and Hausdorff the measures $\mF{i}$ 
are also Radon measures with full support. 
Note that if $\mF0$ is a finite measure with mass $|\mF0|$ then all $\mF{i}$ have the same total mass since $|\mG{i}|=1$ for all $i \ge 1$.

For the rest of the paper a space $F_i$ will be a member of a Barlow-Evans sequence with all the associated components assumed to exist. For any $i = 0,1,\ldots$ we shall denote the $L^{2}$ norm on functions over $F_i$ by $\| \cdot \|_i$. Below $F_{\infty}$ and $\mF{\infty}$ will be defined and this convention will apply to them as well. These norms should not be confused for the $L^{p}$ norm which are not used in this paper except for $p=2$. 

If $f$ is a function on $F_i$ then $\pi_i^{*}f$ is a function on $F_{i-1} \times G_i$ defined by 
$$\pi_i^{*}f = f \circ \pi_i.$$
Similarly for $\phi_i^{*}$ and $\psi_i^{*}$. We shall need the following technical statement to properly describe the function spaces that will be used later.

\begin{lemma}\label{lem:compact}
Let $A \subset F_i$. Then $A$ is compact if and only if $\pi_i^{-1}(A)$ is compact. Also $B\subset F_{i-1}$ is compact if and only if $\phi_i^{-1}(B)$ is compact.
\end{lemma}

\begin{proof}
The result follows from the topologies of $F_i$ and $F_{i-1} \times G_i$ being related through $\pi_i$ and a basic compactness argument.
\end{proof}

We will use $C_0[X]$ to denote the space of continuous functions with compact support.

\begin{corollary}\label{cor:cptsupp}
For all $i \ge 1$, if $f$ is a function on $F_i$, then 
$$ f \in C_0[F_i]\ \text{if and only if}\ \pi_i^{*}f \in C_0[F_{i-1} \times G_i].$$
\end{corollary}

\begin{proof}
The equivalence of continuity is immediate from the quotient topology on $F_i$. The equivalence of the compact support claims follows from Lemma \ref{lem:compact}.
\end{proof}

Following Definitions~\ref{def:Fi} and \ref{def:muFi}, we   consider the spaces $L^{2}(F_i,\mF{i})$ with the norms $\| \cdot \|_i$ for all $i$ on which we now define quadratic forms, which will be shown in Theorem \ref{thm:EiDF} to be Dirichlet forms. 

\begin{definition}\label{def:DF}
Given a regular Dirichlet form $(\C{E}_0,\C{F}_0)$ on $L^{2}(F_0,\mF0)$ 
with a core $\F{F}_0\subset C_0[F_0]$
define inductively quadratic forms on $L^{2}(F_i,\mF{i})$ as follows. 

First, we inductively define the cores of continuous functions  
$$\F{F}_i = \left\{f \in C_0[F_i] \left|\ \begin{array}{c} \pis i f(x,g) = \sum_{k=1}^{n} f_k(x)h_k(g),\\ \ f_k \in   \F{F}_{i-1},\ h_k \in C(G_i)\end{array}\right. \right\}$$
and then we define 
\begin{equation}\label{e-key-product}
	\C{E}_i(f,h) := \int_{G_i} \C{E}_{i-1}(\pi^{*}_i f(\cdot , g),\pi^{*}_i h(\cdot , g))d\mG{i}(g), \hspace{1cm} f,g \in \F{F}_i.
\end{equation}
After that we define $\C{F}_i$ as the completion of $\F{F}_i$ in the norm 
$\sqrt{\C{E}_i(\cdot)}+\| \cdot \|_i$. Note that 
\begin{equation*}
\C{F}_i \subset \hat{\C{F}}_i = \left\{ f \in L^{2}(F_i,\mF{i}) \left| \begin{array}{c} \pi^{*}_i f(\cdot, g) \in \C{F}_{i-1}\ \text{for}\ \mG{i}-a.e.\ g \in G_i\ \text{and} \\ \int_{G_i} \C{E}_{i-1}(\pi^{*}_i f(\cdot , g), \pi^{*}_i f(\cdot , g))d\mG{i}(g) < \infty \end{array} \right.\right\}
\end{equation*}
and so $\C{E}_i$ is well defined on $\C{F}_i$ provided that $(\C{E}_i,\F{F}_i)$ is closable (see Theorem \ref{thm:EiDF}).

The measurability of $\C{E}_{i-1}(\pis{i}f(\cdot,g),\pis{i}f(\cdot,g))$ as a function of $g$ follows from the lower semi-continuity of the map $f \mapsto \C{E}_{i-1}(f,f)$ for $f \in \C{F}_{i-1}$.
\end{definition}

The relationship between $\C{F}_i$ and $\hat{\C{F}}_i$ will be a delicate one where in many instances it will be possible to prove equality. It is not obvious how to do so in complete generality though.%\footnote{Unless you see a way to do so.}

Note that $\pis{i} f$ is a function in two variables, one along $F_{i-1}$ and another along $G_i$. So this definition can be read as applying $\C{E}_{i-1}$ to $\pis{i}f$ for almost every element of $G_i$ and then integrating over $G_i$. Before examining the properties of $(\C{E}_{i},\C{F}_i)$ to ensure that it is really a Dirichlet form we first verify that it is well defined.

%---------\footnote{I did not edit beyond this line  }

\begin{lemma}\label{lem:uniformdense}
If $(\C{E}_{i-1},\C{F}_{i-i})$ is regular for $i \ge 1$, then $\F{F}_i$ is a dense subalgebra  of $C_0[F_i]$.
\end{lemma}

\begin{proof}
Since $\F{F}_i$ consists of functions on $F_i$ whose pull back to $F_{i-1}\times G_i$ are continuous with compact support then by Cor \ref{cor:cptsupp} $\F{F}_i \subset C_0[F_i]$. Density follows from an application of the Stone-Weierstrass Theorem for locally compact spaces \cite[Chapter V, Cor. 8.3]{Conway}. Note that $\F{F}_i$ is an algebra of real valued functions so it remains only to show that for all $x \in F_i$ there exists a $f \in \F{F}_i$ such that $f(x) \neq 0$ and that $\F{F}_i$ separates points. Since $\C{E}_{i-1}$ is a regular Dirichlet form there exists $f \in \C{F}_{i-1} \cap C_0[F_{i-1}]$ so that $f(\phi(x)) > 0$ and $\phis{i}f \in \F{F}_i$.

Let $z_1,z_2 \in F_i$ be distinct points. Then there exists $(x_k,g_k) \in F_{i-1} \times G_i$ for $k=1,2$ such that $\pi(x_k,g_k) = z_k$ Because $z_1 \neq z_2$ it follows that $x_1 \neq x_2$ or $g_1 \neq g_k$, this is an inclusive ``or'' so it is possible that both coordinates a distinct. If $x_1 \neq x_2$ then there exists a $f\in \C{F}_{i-1} \cap C_0[F_{i-1}]$ such that $f(x_1) \neq f(x_2)$. In this case $\phis{i}f \in \F{F}_i$ is a separating function for the points $z_1$ and $z_2$. If $x_1 = x_2$ but $g_1 \neq g_2$ there are two sub-cases $x_k \in B_i$ or $x_k \not\in B_i$. If $x_k \in B_i$ then this forces $g_1=g_2$ so this case cannot happen by the structure of a Barlow-Evans sequence. Suppose then that $x_1 = x_2 \not \in B_i$ and $g_1 \neq g_2$. Such a combination of $x_k$ and $g_k$ imply that $x_1$ is in some open connected component of $F_{i-1} \setminus B_i$, call it $S$. By the regularity of $\C{E}_{i-1}$ there exists $f \in \C{F}_{i-1} \cap C_0[F_{i-1}]$ that is positive at $x_1$ and zero on $S^{C} \subset F_{i-1}$. By Urysohn's Lemma there exists $h \in C_0[G_i]$ such that $h(g_1) = 0$ and $h(g_2)=1$. Then $f(x)h(g)$ is zero on $B_i \times G_i$ so it is the lift of a continuous compactly supported function on $F_i$ which by construction is in $\F{F}_i$. 
\end{proof}

\begin{theorem}\label{thm:EiDF}
If $(\C{F}_0,\C{F}_0)$ is a regular Dirichlet form with core $\F{F}_0 \subset C_0[F_0]$, then   $(\C{E}_i,\F{F}_i)$ are closable forms whose closures are the regular Dirichlet forms $(\C{E}_i,\C{F}_i)$  for all $i \ge 0$. Moreover, if $(\C{E}_0,\C{F}_0)$ is strongly local,  then  $(\C{E}_i,\C{F}_i)$ are strongly local as well.  
%If $(\C{E}_0,\C{F}_0)$ is a regular Dirichlet form, then the $(\C{E}_i,\C{F}_i)$ are also regular Dirichlet forms with cores $\F{F}_i$ for all $i \ge 0$. Moreover, if $\C{E}_0$ is local then $(\C{E}_i,\C{F}_i)$ is local as well.
\end{theorem}

The proof of this theorem is standard and is only sketched below. We also present some intuitive arguments illustrate the situation. 
One feature of a Barlow-Evans sequence that makes the proof of this theorem more complicated is that the domains $\C{F}_i$ are not nested as subsets of the same background set. 
The perspective of nested subspaces   will take the notation of Section \ref{sec:projlap}, and then it involves the use of the projective limit of the $F_i$ and $\mF{i}$, which are not necessary for the proof of this theorem.

\begin{proof}
%We proceed by induction. The base case is the first hypothesis of the theorem. Assume that $(\C{E}_{i-1},\C{F}_{i-1},F_{i-1},\mF{i-1})$ is a regular Dirichlet form with $\F{F}_{i-1}$ as a core. Definition \ref{def:DF} already defines $(\C{E}_i,\C{F}_i)$ as a bilinear, non-negative, and Markovian quadratic form with $\F{F}_i \subset \C{F}_i$. By Lemma \ref{lem:uniformdense} $\F{F}_i$ is uniformly dense in $C_0[F_i]$ which is in turn dense in $L^{2}(F_i,\mF{i})$. Since $\F{F}_i \subset \C{F}_i$ we have that $\C{F}_i$ is dense in $L^{2}(F_i,\mF{i})$.  So provided $(\C{E}_i,\C{F}_i)$ is the closure of $(\C{E}_i,\F{F}_i)$ we have that it is regular. 

We proceed by induction. The base case is the first hypothesis of the theorem. The hypothesis that $\F{F}_0$ is a core for $(\C{E}_0,\C{F}_0)$ is automatically satisfied if $(\C{E}_0,\C{F}_0)$ is a regular Dirichlet form.

Assume that $(\C{E}_{i-1},\F{F}_{i-1})$ is a closable bilinear form and that $(\C{E}_{i-1},\C{F}_{i-1})$ is its smallest closed extension, or closure, which  is a Dirichlet form. Definition \ref{def:DF} already defines $(\C{E}_i,\F{F}_i)$ as a bilinear, non-negative, and Markovian quadratic form, which is closable. This easily follows from the product structure in the right hand side of formula \eqref{e-key-product}, and the fact that 
restricting a closable form to a subspace is  a closable form. 
The
standard references are \cite{CF,FOT} and \cite[Section V.2]{BH} on the products of Dirichlet forms.

If $(\C{E}_{i-1},\C{F}_{i-1})$ is a regular Dirichlet form then by Lemma \ref{lem:uniformdense} $\F{F}_i$ is a   dense sub algebra of $C_0[F_i]$ in the uniform topology. So by standard arguments $\F{F}_i$ is a dense subset of $L^{2}(F_i,\mF{i})$ so $(\C{E}_i,\C{F}_i)$ is densely defined. Also by this lemma we have that $\C{F}_i\cap C_0[F_i]$ is uniformly dense in $C_0[F_i]$. Also by definition of $\C{F}_i$ as the closure in the $\sqrt{\C{E}_i(\cdot,\cdot)} + \| \cdot \|_i$ metric of $\F{F}_i \subset C_0[F_i]$ we have that $(\C{E}_i,\C{F}_i)$ is a regular Dirichlet form with $\C{F}_i \subset \hat{\C{F}}_i$.  

Assume that $(\C{E}_{i-1},\C{F}_{i-1})$ is strongly local. Let $u,v \in \F{F}_i$ have disjoint supports then $\pis{i}(u)(x,g)$ and $\pis{i}(v)(x,g)$ will also have disjoint supports. Consider 
$$ \C{E}_i(u,v) = \int_{G_i} \C{E}_{i-1}(\pis{i}u,\pis{i}v)\mG{i}.$$
Since $\pis{i}u$ and $\pis{i}v$ are continuous functions on $F_{i-1} \times G_i$ we know that for a given $g \in G_i$ that as functions of $x \in F_{i-1}$ that $\pis{i}u(x,g)$ and $\pis{i}v(x,g)$ have disjoint supports (Lemma \ref{lem:compact}). Thus for all $g \in G_i$ $\C{E}_{i-1}(\pis{i}u,\pis{i}v) = 0$ by locality and consequently $\C{E}_i(u,v) = 0$. Since $\F{F}_i$ is a core for $(\C{E}_i,\C{F}_i)$ it is local if $(\C{E}_{i-1},\C{F}_{i-1})$ was. See also   Theorem 3.1.2 and Problem 3.1.1 in \cite{FOT}.
\end{proof}

\begin{remark}
While it will be useful to be able to characterize elements of $\pis{i}\C{F}_i$ as elements of $\pis{i}\hat{\C{F}}_i$ it will be important to remember that these two spaces are not in general the same. 
\end{remark}

The following is a precursor to the nesting of $\C{F}_i$ that will be further developed in the next section. 

\begin{corollary}\label{cor:finitecompat}
The domains of the Dirichlet forms $(\C{E}_i,\C{F}_i)$ are compatible
 in the sense that 
$$\phis{i}\C{F}_{i-1} \subset \C{F}_i.$$
\end{corollary}

\begin{proof}
Let $f \in \C{F}_{i-1}$, then $\C{E}_i(\phis{i}f) = \int_{G_i}\C{E}_{i-1}(f)\mG{i} = \C{E}_{i-1}(f) < \infty.$ Also $\phis{i}f \in L^{2}(F_i)$ because $G_i$ is compact and $\mG{i}$ is a probability measure. 
\end{proof}

%\begin{corollary}\label{cor:stronglocal}
%If $(\C{E}_{i-1},\C{F}_{i-1})$ is strongly local in addition to the other hypotheses of Theorem \ref{thm:EiDF} then $(\C{E}_i,\C{F}_i)$ is also %strongly local. 
%\end{corollary}

%\begin{proof}\end{proof}

\section{Projective Limits}\label{sec:projlimits}

The construction that is considered in this paper is a means of constructing state spaces for  symmetric diffusions via projective limits. That is, taking limits along compatible sequences of topological spaces and producing a limit topological space. More work is required to construct compatible sequences of metrics, measures, and Dirichlet forms. Barlow and Evans \cite{BE04} considered this construction as a way to produce exotic state spaces for Markov processes. Then \cite{KKPSS} specialized Barlow and Evans' work to Laakso spaces \cite{Laakso2000}.

\begin{definition}\label{def:projlimits}
Let $\prod_{i=1}^{\infty} F_i$ have the product topology. For a Barlow-Evans sequence the projective limit $\lim_{\leftarrow} F_i$, denoted by $F_{\infty}$ is a subset of $\prod_{i=1}^{\infty} F_i$ with the subspace topology such that for any $(x)_{i=1}^{\infty} \in F_{\infty}$ $\phi_i(x_i) = x_{i-1}$ and the canonical projections $\Phi_j: \prod F_i \rightarrow F_j$ restrict to $F_{\infty}$ and have the consistency property:
$$ \phi_j \circ \Phi_j = \Phi_{j-1},\ j \ge 1.$$
\end{definition}

Note that the topology on $F_{\infty}$ is Hausdorff and second countable. It is also locally compact \cite[IX Sec 4]{BourbakiIntegrationII}. We now turn to defining a measure on $F_{\infty}$.

\begin{proposition}{\cite[IX Sec 4]{BourbakiIntegrationII}}\label{prop:projlimitmeas}
There exists a unique measure on $F_{\infty}$ denoted $\mF{\infty}$ if the masses of $\mF{i}$ are uniformly bounded. Then $\mF{\infty}$ satisfied 
\begin{equation}\label{eq:measconsist}
	 \mF{i}(A) = \mF{\infty}(\Phi^{-1}_i(A))
\end{equation}
for all $A$ that are $\mF{i}-$measurable. Further more, if the $\mF{i}$ are Radon measures so is $\mF{\infty}$. 
\end{proposition}

The existence and uniqueness claim in Theorem 2 in IX Section 4 in \cite{BourbakiIntegrationII}. The claim about the Radon property follows from Propositions 1-3.

\begin{corollary}
If $\mF{0}$ is $\sigma-$finite then there exists a unique $\mF{\infty}$ on $F_{\infty}$ which satisfies Equation \ref{eq:measconsist}.
\end{corollary}

\begin{proof}
Since $\mF{0}$ is $\sigma-$finite there exists a partition of $F_0$ such that each element of the partition has finite measure. By partitioning $F_0$ it follows that the lift of the partition to $F_i$ is also a partition where each piece  has finite mass sets then $(F_i,\mF{i})$ is a $\sigma-$finite measure space. Each member of the partition of $F_i$ has the same measure as the corresponding member of the partition of $F_0$, so the masses stay bounded in $i$. Apply Proposition \ref{prop:projlimitmeas} on each member of the partition starting at $F_0$ and then take $\mF{\infty}$ to be their sum.
\end{proof}

We shall often have probability measures on $F_i$ so that it will be possible to consider directly the limit measure space $(\lim_{\leftarrow}F_i,\mF{\infty})$ rather than using this Corollary. Note that the $\Phi^{*}_i$ are \B{R}-linear maps from Borel functions on $F_i$ to Borel functions on $F_{\infty}$. 

\begin{proposition}\label{prop:InftyUnifDense}
Let $clos_{uniform}$ represent the closure operation in the uniform norm then 
$$C_0[F_{\infty}] = clos_{uniform} \left\{ \bigcup_{i=0}^{\infty} \Phi^{*}_i C_0[F_i]\right\}.$$
\end{proposition}

\begin{proof}
As in the proof of Lemma \ref{lem:uniformdense} using the Stone-Weierstrass Theorem. 
\end{proof}

\section{Projections and Laplacians}\label{sec:projlap}

Having constructed Dirichlet forms on the approximating spaces, $F_i$, in Section \ref{sec:def} we now turn to constructing a Dirichlet form over the limit space, $F_{\infty}$ which was constructed in Section \ref{sec:projlimits}. Recall that the $L^{2}(F_M,\mF{M})$ norm is denoted by $\| \cdot \|_M$ for $M=0,1,2,\ldots , \infty$. The existence of projective limits of Dirichlet spaces ($L^{2}$ space equipped with a Dirichlet form and its domain) is briefly discussed in \cite{BH}. We develop the existence for the sake of the accompanying notation which is then used to describe the decompositions in Theorem \ref{thm:fcnspacedecomp}. The decompositions rely on the specific structure of the equivalence relations used in defining a Barlow-Evans sequence and are not a general feature of projective systems of Dirichlet spaces. 

\begin{definition}
Given a Barlow-Evans sequence let $\C{E}_{\infty}$ be the quadratic form on $F_{\infty} = \lim_{\leftarrow} F_i$ defined by 
$$\C{E}_{\infty}(\Phi^{*}_i u, \Phi^{*}_i u) = \C{E}_i(u,u)$$
for all $u \in \C{F}_i$ for all $i \ge 1$. The domain of $\C{E}_{\infty}$ is 
$$\C{F}_{\infty}= clos \left\{ \bigcup_{i=0}^{\infty} \Phi^{*}_i \F{F}_i \right\}.$$
The closure is in the $\C{E}_{\infty}^{1/2} + \| \cdot \|_{\infty}$ metric. 
\end{definition}

As in Section \ref{sec:def} we must show that this definition is suitable. Specifically that $\C{E}_{\infty}$ is closable and that the minimal closed extension is $(\C{E}_{\infty},\C{F}_{\infty})$. In the manner of Section \ref{sec:def} we define 
$$\hat{\C{F}}_{\infty}= clos \left\{ \bigcup_{i=0}^{\infty} \Phi^{*}_i \hat{\C{F}}_i \right\}.$$
The possible equality of $\C{F}_{\infty}$ and $\hat{\C{F}}_{\infty}$ will not be addressed in any generality. For Laakso spaces it is known that they are the same, see \cite{Steinhurst2010} and Subsection \ref{ssec:Laakso}.

By Corollary \ref{cor:finitecompat}, the $\Phi_i^{*}\F{F}_i$ are increasing linear subspaces of $L^{2}(F_{\infty},\mF{\infty})$ and $\bigcup_{i \ge 0} \Phi^{*}_i\F{F}_i$ is a dense linear subspace of $L^{2}(F_{\infty},\mF{\infty})$. Notice that by the relationship $\Phi^{*}_i = \Phi_{i+1}^{*} \circ \phi_i^{*}$ we have that $\C{E}_{\infty}(\Phi^{*}_i u) = \C{E}_{\infty}(\Phi_{i+1}^{*} \circ \phi_i^{*} u)$ for all $u \in \F{F}_i$. From this we see that the quadratic form $(\C{E}_{\infty}, \bigcup_{i \ge 0} \Phi^{*}_i \F{F}_i)$ is well-defined.

\begin{theorem}\label{thm:InftyRegularity}
If $(\C{E}_0,\C{F}_0)$ is a regular Dirichlet form then the pair $(\C{E}_{\infty},\C{F}_{\infty})$ is a regular Dirichlet form. Furthermore, if $\C{E}_0$ is strongly local then $\C{E}_{\infty}$ is strongly local as well.
\end{theorem}

\begin{proof}
On $\bigcup_{i \ge 0} \Phi^{*}_i \F{F}_i$ the form $\C{E}_{\infty}$ is linear, positive, and has the Markovian property. Suppose for the moment that $(\C{E}_{\infty}, \bigcup_{i \ge 0} \Phi^{*}_i \F{F}_i)$ is closable. Linearity and positivity are maintained in the closure with respect to the $\C{E}_{\infty}^{1/2} + \| \cdot \|_{\infty}$ metric. By Theorem 3.1.1 of \cite{FOT} the Markovian property extends to the smallest closed extension of $(\C{E}_{\infty},\bigcup_{i \ge 0} \Phi^{*}_i \F{F}_i)$ which is $\C{F}_{\infty}$ by virtue of it being the closure in the metric induced by the form $\C{E}_{\infty}$ itself.

To show that $(\C{E}_{\infty},\bigcup_{i \ge 0} \Phi^{*}_i \F{F}_i)$ is closable, one can employ the standard monotonicity methods \cite[Theorem S.14, page 373]{RS80}.  To make this construction more concrete, note that all of the $G_i$ are compact probability spaces, and $\prod_{i=1}^{\infty} G_i$ is also a compact probability space with the product topology. For $u \in \bigcup_{i \ge 0} \Phi^{*}_i \F{F}_i$ there exists a $j \in \mathbb{N}$ such that $u = \Phi^{*}_j v$ for some $v \in \F{F}_j$. Let $g_j \in \prod_{i=1}^{j} G_i$ and $g_{j+} \in \prod_{i=j+1}^{\infty} G_i$. Since $\pi_1$ acts on $F_0 \times G_1 \times \cdots \times G_j$ by taking the first two coordinates and returning an elements of $F_1 \times G_2 \times \cdots \times G_j$ upon which $\pi_2$ has a similar action we can compose $\pi_i$ let $\pi_{j,1} = \pi_j \circ \cdots \circ \pi_1$. Then we can take advantage of the structure of Barlow-Evans sequence:
\begin{align*}
\C{E}_{\infty}(u) &= \C{E}_j(v)\\
& = \int_{\prod_{i=1}^{j}G_i} \C{E}_0(\pi_{j,1}^{*}(v)(x,g_j)d\mu_{\prod_{i=1}^{j} G_i}(g_j)\\
&= \int_{\prod_{i=j+1}^{\infty}G_i} \left(  \int_{\prod_{i=1}^{j}} \C{E}_0(\pi_{j,1}^{*}(v)(x,g_j)d\mu_{\prod_{i=1}^{j} G_i}(g_j) \right)\ d\mu_{\prod_{i=j+1}^{\infty} G_i} (g_{j+})\\
&= \int_{\prod_{i=1}^{\infty}G_i} \C{E}_0((\pi_{j,1}^{*}(v))'(x,g_j)\ d\mu_{\prod_{i=1}^{\infty}G_i}(g_j)
\end{align*}
where $(\pi_{j,1}^{*}(v))'(x,g_j)$ is $\pi_{j,1}^{*}(v)(x,g_j)$ extended to a function of $x$, $g_j$, and $g_{j+}$ by declaring it constant in $g_{j+}$. For functions in $\bigcup_{i \ge 0} \Phi^{*}_i\F{F}_i$ the composition $\pi_{\infty,1}^{*}$ eventually stabilizes at some finite $j$ so by the above we can write
$$\C{E}_{\infty}(u) =  \int_{\prod_{i=1}^{\infty}G_i} \C{E}_0(\pi_{\infty,1}^{*}(u)(x,g_{1+})) d\mu_{\prod_{i=1}^{\infty}G_i}(g_j).$$
This is an analogous definition for $\C{E}_{\infty}$ as was made for $\C{E}_i$ as constructed from $\C{E}_{i-1}$ in Theorem \ref{thm:EiDF}. 

By Theorem 3.1.2 \cite{FOT}, a local closable Markovian symmetric form $(\C{E},\C{F})$ on $L^{2}(X,\mu)$ has the local property on its smallest closed extension if it has a core that is a dense subalgebra of $C_0[X]$ and every compact set, $K$, has a pre compact open neighborhood, $G$, such that there exists $u \in \C{F}$ such that $u(x) = 1$ for all $x \in K$ and $u(x)=0$ for all $x \in X \setminus G$. The existence of such a core is exhibited by choosing it to be $\bigcup_{i \ge 0} \F{F}_i$ since all such functions are continuous by definition and as was remarked above it is a dense sub algebra of $C_0[F_{\infty}]$. Let $K \subset F_{\infty}$ be compact. Set $K' = \Phi_0(K) \subset F_0$. Since $(\C{E}_0,\C{F}_0)$ is a local regular Dirichlet form there exists an open set $G' \supset K'$ and $u' \in \C{F}_0$ such that $u'(x) = 1$ for $x \in K'$ and $u'(x) = 0$ for $x \in F_0 \setminus G'$. Let $G = \Phi^{*}_0(G')$ and $u = \Phi^{*}_0 u'$. $G$ is pre compact since $F_{\infty}$ is locally compact.
%\footnote{Check this, I don't remember off the top of my head if this is correct.}
Because $K \subset \Phi^{-1}_0(K')$ we have $u(x) = 1$ for all $x \in K$ and similarly $u(x) = 0$for all $x \in F_{\infty} \setminus G$. As in 
%Corollary \ref{cor:stronglocal} 
Theorem~\ref{thm:EiDF}, 
Theorem 3.1.2 and Problem 3.1.1 of \cite{FOT} imply that if $(\C{E}_0,\C{F}_0)$ is strongly local then $(\C{E}_{\infty},\C{F}_{\infty})$ is as well.

The regularity of $(\C{E}_{\infty},\C{F}_{\infty})$ comes from the fact that $\C{F}_{\infty}$ is defined as the closure of a set of continuous functions and Lemma \ref{lem:FInftyUnifDense} which says that those continuous functions are also uniformly dense in $C_0[F_{\infty}]$. 
\end{proof}

Logically the following lemma ccomes before the above theorem. The only reason it is placed here is for the notation of $\pi_{\infty,1}$ discussed in the theorem's proof. 

\begin{lemma}\label{lem:FInftyUnifDense}
If $(\C{E}_0,\C{F}_0)$ is regular then $\C{F}_{\infty}\cap 
C_0[F_{\infty}]$ is a dense subalgebra of $C_0[F_{\infty}]$.
\end{lemma}

\begin{proof} 
It is clearly a subalgebra. 
%Since $F_{\infty}$ is given the minimal topology such that all the $\Phi_i$ are continuous the topologies of the $F_i$ lifted to $F_{\infty}$ form a basis for the topology on $F_{infty}$ thus we need to ensure that $\C{F}_{\infty}$ 
We use the Stone-Weierstrass Theorem in the same manner as in Lemma \ref{lem:uniformdense}. Choose $z_1 \neq z_2 \in F_{\infty}$. Then $\pi_{\infty,1}(z_1)$ and $\pi_{\infty,1}(z_2)$ as functions on $F_0 \times G_1 \times \cdots$ differ in at least one coordinate. If that coordinate is $F_0$ then the same argument as in Lemma \ref{lem:uniformdense} can be used again to show a pair of functions separating these two points. If the first coordinate in which a difference occurs is $G_j$ then by Theorem \ref{thm:EiDF} implies that $\C{E}_{j-1}$ is regular and then the proof of Lemma \ref{lem:uniformdense} again shows that there exists a pair of separating functions. Hence by Stone-Weierstrass we have the proof. 
\end{proof}

\begin{theorem}
If $\Delta_i$ is the Laplacian generated by $\C{E}_i$ and $\Phi_j\ :\ \lim_{\leftarrow} F_i \rightarrow F_j$ the continuous projection form the projective limit construction. Then 
$$\Phi_{i-1}^{*} Dom(\Delta_{i-1}) \subset \Phi_i^{*} Dom(\Delta_i)\ \forall i \ge 0.$$
\end{theorem}

\begin{proof}
For a general Dirichlet form $(\C{E},\C{F})$ with generator $\Delta$, $h$ is in $Dom(\Delta)$ if and only if there exists $f \in L^{2}$ such that 
$$\C{E}(h,v) = \langle f,v\rangle_{L^{2}}$$
for any $v \in \C{F}$, and in this situation $\Delta h = f$. It is sufficient to check that if $u \in Dom(\Delta_{i-1})$ then $\phi_i^{*}u \in Dom(\Delta_i)$. Since $\C{F}_i \subset \hat{\C{F}}_i$ we have that $\pis{i}v(\cdot,g) \in \C{F}_{i-1}$ for almost every $g$. Let $u \in Dom(\Delta_{i-1})$ and $v \in \C{F}_i$. Then
\begin{align*}
	\C{E}_i(\phi_i^{*},v) &= \int_{G_i} \C{E}_{i-1}(\pis{i}\phi^{*}_i u, \pis{i} v)(g)\ d{\mG{i}}(g)\\
	&=  \int_{G_i} \C{E}_{i-1}(u,\pis{i} v)(g)\ d{\mG{i}}(g)\\
	&=  \int_{G_i} \int_{F_{i-1}} \Delta_{i-1}u(x,g)\ \pis{i}v(x,g)\ d{\mF{i-1}}(x)\ d{\mG{i}}(g)\\
	&= \int_{F_{i-1}\times G_i} \Delta_{i-1}u(x,g)\ \pis{i}v(x,g)\ d(\mF{i-1}\times \mG{i})(x,g)\\
	&= \int_{F_i} (\phi^{*}_i \Delta_{i-1}u(x)\ v(x)\ d\mF{i}(x).
\end{align*}
Thus $\phi^{*}_i(\Delta_{i-1} u) = \Delta_i \phi^{*}_i u.$ So $\phi^{*}_i u \in Dom(\Delta_i)$.
\end{proof}

\begin{definition}
For $i \ge 1$, given a Borel measurable $f\ :\ F_i \rightarrow \B{R}$ define the projections $\tilde{\C{P}}_i\ :\ L^{2}(F,\mF{i}) \rightarrow L^{2}(F_{i-1},\mF{i-1})$ and $\C{P}_i\ :\ L^{2}(F_i,\mF{i}) \rightarrow L^{2}(F_i,\mF{i})$ by 
$$\tilde{\C{P}}_i(f)(x) =  \int_{G_i} (\pis{i}f)(x,g) d\mG{i}(g)$$
and
$$\C{P}_i(f)(x) = \phi^{*}_i \left( \int_{G_i} (\pis{i}f)(x,g) d\mG{i}(g) \right) = \phi^{*}_i \tilde{\C{P}}_i(f)(x).$$
These projections can be restricted to have domains $C_0[F_i]$ or $\C{F}_i$ as subspaces of $L^{2}(F_i,\mF{i})$. The domain will be made clear in each context. 
\end{definition}

The integral in this definition maps a function on $F_{i-1} \times G_i$ to a function on $F_{i-1}$ so that $\C{F}_i$ takes functions on $F_i$ and returns another function on $F_i$. Note that $\C{P}_i(f)(x) = f(x)$ for $x \in B_i$ because $\pis{i}f(x,g)$ is constant overall values of $g$ if $x \in B_i$. On the other hand $\tilde{\C{P}}_i$ can be composed to project down several levels, say from $i$ to $i=3$. Let $\Pi_i \ (\Phi^{*}_i)^{-1}proj_{\Phi^{*}_i(L^{2}(F_i,\mF{i}))}$, where $proj_X$ is the orthogonal projection in $L^{2}(F_{\infty},\mF{\infty})$ onto a closed subspace $X$, which is the left inverse of $\Phi^{*}_i$. The families $\C{P}_i$, $\tilde{\C{P}}_i$, and $\Pi_i$ satisfy the following relation for $f \in L^{2}(F_i,\mF{i})$:
$$\Pi_{i-1} \circ \Phi^{*}_i(f) = \tilde{\C{P}}_i(f).$$
The map $\Pi_i$ has, for functions in $\bigcup_{i \ge 0} \Phi^{*}_i \C{F}_i$, a nice explicit form. Suppose $u \in \bigcup_{i \ge 0} L^{2}(F_i,\mF{i})$ then there exists a $j \in \B{N}$ and a $v \in L^{2}(F_j,\mF{j})$ such that $u = \Phi^{*}_j v$. Then
$$\Pi_i u = \left\{\begin{array}{lr} \tilde{\C{P}}_j \circ \cdots \circ \tilde{\C{P}}_{i+1} v & j>i \\ v & j=i \\ \phi^{*}_i \circ \cdots \circ \phi^{*}_{j+1} v & j<i.\end{array} \right.$$

\begin{definition}\label{def:projkernel}
Since $C_0[F_i]$ and $\C{F}_i$ have natural injections into $L^{2}(F_i,\mF{i})$ we can set the following notation:
\begin{align*}
ker(\C{P}_i|_{L^{2}(F_i,\mF{i})}) &= \C{L}_i\\
ker(\C{P}_i|_{C_0[F_i]} )&= \C{C}_i\\
ker(\C{P}_i|_{\C{F}_i} )&= \C{F}'_i
\end{align*}
\end{definition}

The following three lemmas describe the behaviors of the projection $\C{P}_i$ on each of its three domains of interest.

\begin{lemma}\label{lem:projL2}
Let $\C{P}_i$ be defined on $L^{2}(F_i,\mF{i})$ as above. Then 
$$L^{2}(F_i,\mF{i}) = \phi^{*}_i(L^{2}(F_{i-1},\mF{i-1}) \oplus \C{L}_i.$$
Moreover, $h \in \C{L}_i$ if and only if $\phi^{*}_ih(x,g)$ satisfies
$$\int_{G_i} \phi^{*}_ih(x,g)\ d\mG{i} = 0$$
for $\mF{I-1}$-almost every $x \in F_{i-1}$.
\end{lemma}

\begin{proof}
The operators $\C{P}_i$ is an orthogonal projection operators. The eigenspace corresponding to the eigenvalue $1$ is precisely those functions for which $\int_{G_i} \pis{i}f(x,g)\ d\mG{i} = \pis{i}f(x,g)$ for all $x \in F_{i-1}$ and $g \in G_i$. These functions are in $\phi^{*}_i(L^{2}(F_i,\mF{i}))$. The orthogonal complement is then the kernel of the projection.
\end{proof}

\begin{lemma}\label{lem:projC}
Let $\C{P}_i$ be defined on $C_0[F_i]$. Then 
$$C_0[F_i] = \phi^{*}_i(C_0[F_{i-1}]) \oplus \C{C}_i.$$
Moreover, $h \in \C{C}_i$ if and only if $\pis{i}h(x,g)$ satisfies
$$\int_{G_i} \pis{i}h(x,g)\ d\mG{i} = 0$$
for all $x \in F_{i-1}$.
\end{lemma}

\begin{proof}
Claim: $\tilde{\C{P}}_i(C_0[F_i]) \subset C_0[F_{i-1}]$. Let $f \in C_0[F_i]$. Since $\pis{i}f(x,g) \in C_0[F_{i-1}\times G_i]$ this reduces to whether continuity is preserved when integrating over $G_i$, that is if 
$$\int_{G_i} \pis{i}f(x,g)\ d\mG{i}(g)$$
is continuous in $x \in F_{i-1}$. Bu since $\pis{i}f$ is a compactly supported continuous function it is bounded and an application of the Lebesgue Dominated Convergence Theorem provides the continuity. Now note that $\C{C}_i = \C{L}_i \cap C_0[F_i]$ and $\phi^{*}_i(C_0[F_{i-1}]) = \phi^{*}_i(L^{2}(F_{i-1},\mF{i-1}) \cap C_0[F_i])$.
\end{proof}

\begin{lemma}\label{lem:projF}
Let $\C{P}_i$ be defined on $\C{F}_i$. Then 
$$\C{F}_i = \phi^{*}_i(\C{F}_{i-1}) \oplus \C{F}'_i.$$
Moreover, $h \in \C{F}'_i$ if and only if $\pis{i}h(x,g)$ satisfies
$$\int_{G_i} \pis{i}h(x,g)\ d\mG{i} = 0$$
for $\C{E}(\cdot) + \| \cdot \|_i^{2}$-almost every $x \in F_i$. Moreover, the core $C(F_i) \cap \C{F}_i$ of the Dirichlet form $(\C{E}_i,\C{F}_i)$ has the same decomposition.
\end{lemma}

\begin{proof}
On $\C{F}_i$, $\C{P}_i$ is the orthogonal projection. Its range by the same arguments as in Lemma \ref{lem:projL2} is $\phi^{*}_i(\C{F}_{i-1})$ which has also for the same reasons kernel $\C{F}'_i$. The core decomposes as a consequence of the first claim of this lemma and Lemma \ref{lem:projC}.
\end{proof}

\begin{lemma}\label{lem:DeltaInfty}
The generator of $(\C{E}_{\infty},\C{F}_{\infty})$, denoted $\Delta_{\infty}$, is the weak limit of $\Phi^{*}_i \Delta_i \Pi_i$ that is 
$$\Pi_i(Dom(\Delta_{\infty}) = Dom(\Delta_i) \text{\ and\ } \Delta_i\Pi_i|_{Dom(\Delta_{\infty})} = \Pi_i \Delta_{\infty}$$
for any $i \ge 0$. Furthermore for any $f \in Dom(\Delta_{\infty})$,
$$\lim_{i \rightarrow \infty} \Phi^{*}_i\Delta_i \Pi_i f = \Delta_{\infty} f$$
in $L^{2}(F_{\infty},\mF{\infty})$.
\end{lemma}

\begin{proof}
First $\Delta_{\infty}$ is the unique maximal self-adjoint operator on $L^{2}(F_{\infty},\mF{\infty})$ such that for all $f \in Dom(\Delta_{\infty}) \subset \C{F}_{\infty}$ and $g \in \C{F}_{\infty}$ that
$$ \langle \Delta_{\infty} f,g\rangle = \C{E}_{\infty}(f,g).$$
The first claim is equivalent to $proj_{\Phi^{*}_{i}(L^{2}(F_i,\mF{i}))}Dom(\Delta_{\infty}) \subset \Phi^{*}_i Dom(\Delta_i).$ The opposite inclusion is trivial. Observe that for $f,g \in \bigcup_{i \ge 0} \Phi^{*}_i \F{F}_i$ that 
$$\C{E}_{\infty}( proj_{\Phi^{*}_I(L^{2}(F_i,\mF{i}))^{\perp}}f, proj_{\Phi^{*}_i(L^{2}(F_i,\mF{i}))}g) = 0,$$
and that $\bigcup_{i \ge 0} \Phi^{*}_i \F{F}_i$ is dense in $\C{F}_{\infty}$ so this extends to all of $\C{F}_{\infty}$. Also the projection of elements of $\C{F}_{\infty}$ onto $\Phi^{*}_iL^{2}(F_i,\mF{i})$ are elements of $\Phi^{*}_i\C{F}_i$. Combining these observations we have that for $g \in \Phi^{*}_i\C{F}_i \subset \C{F}_{\infty}$ and $f \in Dom(\Delta_{\infty}) \subset \C{F}_{\infty}$
\begin{align*}
	\C{E}_i(\Pi_if,\Pi_i g) & = \C{E}_{\infty}(proj_{\Phi^{*}_i (L^{2}(F_i,\mF{i}))} f,g)\\
	&= \C{E}_{\infty}(f,g) - \C{E}_{\infty}(proj_{\Phi^{*}_i (L^{2}(F_i,\mF{i}))^{\perp}} f,g)\\
	&= \langle \Delta_{\infty} f,g \rangle_{L^{2}(F_{\infty},\mF{\infty})} - 0\\
	&= \langle  \Delta_{\infty} f, proj_{\Phi^{*}_i (L^{2}(F_i,\mF{i}))}g \rangle_{L^{2}(F_{\infty},\mF{i\infty})}\\
	&= \langle proj_{\Phi^{*}_i (L^{2}(F_i,\mF{i}))} \Delta_{\infty} f, proj_{\Phi^{*}_i (L^{2}(F_i,\mF{i}))}g \rangle_{L^{2}(F_{\infty},\mF{i})}\\
	&= \langle \Pi_i \Delta_{\infty} f, \Pi_i g \rangle_{L^{2}(F_i,\mF{i})}
\end{align*}
since $g = proj_{\Phi^{*}_i (L^{2}(F_i,\mF{i}))} g$. From this we have that $\C{E}_i(\Pi_i f,g') = \langle \Pi_i \Delta_{\infty} f, g' \rangle_{L^{2}(F_i,\mF{i})}$ for all $g' \in \C{F}_i$ hence $\Pi_i f \in Dom(\Delta_i)$ and $\Delta_i \Pi_i = \Pi_i \Delta_{\infty}$ on $Dom(\Delta_{\infty})$.

The convergence in norm of $\Phi^{*}_i \Delta_i \Pi_i f = \Phi^{*}_i \Pi_i \Delta_{\infty} = proj_{\Phi^{*}_i(L^{2}(F_i,\mF{i}))} \Delta_{\infty}$ follows from the fact that $proj_{\Phi^{*}_i(L^{2}(F_i,\mF{i}))} \rightarrow id$ as $L^{2}$ operators.
\end{proof}

\begin{definition}
Let $\F{D}_0' = \Phi^{*}_0 Dom(\Delta_0)$. Then inductively define $\F{D}_i'$ by:
$$\F{D}_i' = \Phi^{*}_i Dom(\Delta_{i}) \cap \F{D}_{i-1}'^{\perp}.$$
The orthogonal compliment is taken in $L^{2}(F_{\infty},\mF{\infty})$. This implies that $\Phi^{*}_i Dom(\Delta_i) = \oplus_{j=0}^{i} \F{D}_j'.$
\end{definition}

\begin{theorem}\label{thm:fcnspacedecomp}
Using the notation of Definition \ref{def:projkernel} we have the following decompositions:
\begin{align*}
	L^{2}(F_{\infty},\mu_{\infty}) &= clos_{L^{2}(F_{\infty},\mu_{\infty})}\left( \Phi^{*}_0 L^{2}(F_0,\mu_{F_0}) \oplus\left( \oplus_{i=1}^{\infty} \Phi^{*}_i \C{L}_i  \right)\right)\\
	C(F_{\infty}) &= clos_{unif}\left( \Phi^{*}_0 C(F_0) \oplus \left( \oplus_{i=1}^{\infty} \Phi^{*}_i \C{C}_i  \right)\right)\\
	\C{F}_{\infty} &= clos_{\C{F}_{\infty}} \left(\Phi^{*}_0  \C{F}_0 \oplus \left( \oplus_{i=1}^{\infty} \Phi^{*}_i \C{F}'_i  \right)\right).
\end{align*}
\end{theorem}

\begin{proof}
By definition $L^{2}(F_{\infty},\mu_{F_{\infty}})$ is the completion of $\bigcup_{i=0}^{\infty} \Phi^{*}_i L^{2}(F_i,\mu_{F_i})$ what is new is the direct sum decomposition. Let $f \in L^{2}(F_1,\mF{1})$ then notice that $f = (f-\C{P}_1 f) + \phis 1 (\tilde{\C{P}}_1 f) \in \C{L}_1 \oplus \phis 1 L^{2}(F_0,\mF{0})$. In general for $f \in L^{2}(F_2,\mF{2})$ we would have
\begin{equation*}
	f = (f- \C{P}_2f) + \phis 2( \tilde{\C{P}}_2 f - \C{P}_1 \tilde{\C{P}}_2 f) + \phis 2 \phis 1(\tilde{\C{P}}_1\tilde{\C{P}}_2 f) \in \C{L}_2 \oplus \phis 2 \C{L}_1 \oplus \phis 2 \phis 1 L^{2}(F_0,\mF{0}). 
\end{equation*}
Continuing by this method we have the direct sum expansion for $L^{2}(F_i,\mF{i})$ for any $i \ge 1$. The $L^{2}(F_{\infty},\mF{\infty})$ limits of these expansions must then be all of $L^{2}(F_{\infty},\mF{\infty})$ since they contain $\bigcup_{i \ge 0} \Phi^{*}_i L^{2}(F_{i},\mF{i})$. The same argument works for $C(F_{\infty})$ and $\C{F}_{\infty}$.
\end{proof}

The domain of $\Delta_{\infty}$ can be decomposed into the direct sum of $\F{D}_i'$, or as $Dom(\Delta_{\infty}) \cap \C{L}_i$ or as $Dom(\Delta_{\infty}) \cap \C{F}_i'$.

\begin{lemma}
The three direct sum decompositions of $Dom(\Delta_{\infty})$ mentioned above agree, that is:
\begin{align*}
	\F{D}'_i = Dom(\Delta_{\infty}) \cap \Phi^{*}_i \C{L}_i = Dom(\Delta_{\infty}) \cap \Phi_i^{*} \C{F}'_i
\end{align*}
for $i \ge 0$ and $\C{L}_0 = L^{2}(F_0)$ and $\C{F}'_0 = \C{F}_0$. Furthermore, the closure of $\Delta_{\infty}|_{\oplus_{i=0}^{\infty} \F{D}'_i}$ $\left( and \ \Delta_{\infty}|_{\bigcup_{i=0}^{\infty} \Phi_i^{*} \C{F}_i} \right)$ in the graph-norm, $\| \cdot \|^{2}_{\infty} + \langle \Delta_{\infty} \cdot , \cdot \rangle$,  is equal to $\Delta_{\infty}$. 
\end{lemma}

\begin{proof}
Because $\C{F}_{i} \subset L^{2}(F_{i})$ we know that $\C{F}'_i \subset  \C{L}_i$. This, together with the fact that $Dom(\Delta_{\infty}) \subset \C{F}_{\infty}$, implies that $Dom(\Delta_{\infty}) \cap \Phi^{*}_i\C{L}_i =Dom(\Delta_{\infty}) \cap \Phi^{*}_i \C{F}_i'$. 

For $f \in Dom(\Delta_{\infty})$ observe that $\Phi_i^{*}\Pi_i f \in \Phi^{*}_i Dom(\Delta_i) \cap \Phi^{*}_i L^{2}(F_i)$ as well as in $L^{2}(F_{\infty})$ so $\Phi^{*}f \rightarrow f$ in $L^{2}(F_{\infty})$ and $\Delta_{\infty}\Phi_i^{*} \Pi_i f \rightarrow \Delta_{\infty} f$ in $L^{2}(F_{\infty})$ as $i \rightarrow \infty$ thus $\Delta_{\infty}$ is the closure of its restriction to $\bigcup_{i=0}^{\infty} \Phi_i^{*} Dom(\Delta_i)$. 
\end{proof}

\section{Main Results}\label{sec:main}

From the discussion in the previous section we can consider $(\Delta_{\infty}, \overline{\F{D}_i'})$ as a densely defined operator on $\Phi^{*}_i L^{2}(F_i)$ as a closed subspace of $L^{2}(F_{\infty},\mF{\infty})$, that is 
\begin{displaymath}
	\overline{\F{D}_i'}^{L^{2}(F_{\infty},\mF{\infty})} = \Phi^{*}_i L^{2}(F_i,\mF{i}) \cap (\Phi^{*}_{i-1} L^{2}(F_{i-1},\mF{i-1}))^{\perp, L^{2}(F_{\infty},\mF{\infty})}.
\end{displaymath}
However since $ \overline{\F{D}_i'}$ can be written as the intersection of two closed subspaces it is itself closed in $L^{2}(F_{\infty},\mF{\infty})$ and we can drop the closure symbol. 

\begin{theorem}\label{thm:specdecomp}
The spectrum of $\Delta_{\infty}$ is given by:
\begin{equation*}
	\sigma(\Delta_{\infty}) = \overline{\bigcup_{i=0}^{\infty} \sigma(\Delta_i)} = \overline{\bigcup_{i=0}^{\infty} \sigma(\Delta_{\infty}|_{\F{D}'_i})}
\end{equation*}
\end{theorem}

\begin{proof}
We begin with the statement that $\sigma(\Delta_n) = \sigma(\Delta_{\infty}|_{\F{D}_n})$ where $\F{D}_n = \oplus_{i=0}^{n} \F{D}_i'$. Since the $\F{D}_i'$ are closed mutually-orthogonal subspaces of $L^{2}(F_{\infty},\mF{\infty})$ $\F{D}_n$ is a closed subspace and $\Delta_{\infty}$ is defined on it. Because $\F{D}_n$ is the direct sum in $L^{2}(F_{\infty},\mF{\infty})$ of only finitely many $\F{D}_i'$ then  $\sigma(\Delta_{\infty}|_{\F{D}_n}) = \bigcup_{i=0}^{n} \sigma(\Delta_{\infty}|_{\F{D}_i'})$. From this the right hand equality in the statement follows. 

Let $z \in \sigma(\Delta_n)$. Then by Lemma \ref{lem:DeltaInfty} $(\Delta_n - z)$ is not invertible on $Dom(\Delta_n)$. Since $(\Delta_{\infty} - z)$ agrees with $(\Delta_n - z)$ on $\Phi_n^{*} Dom(\Delta_n) \subset Dom(\Delta_{\infty})$ we have that $(\Delta_{\infty} - z)$ is not invertible. Hence $\sigma(\Delta_n) \subset \sigma(\Delta_{\infty})$ for all $n \ge 0$. So $\sigma(\Delta_{\infty}) \supset \overline{\bigcup_{i=0}^{\infty} \sigma(\Delta_i)}$. The other containment will take more work. 

Suppose that $z \in \sigma(\Delta_{\infty})$ and $z \not\in \overline{\bigcup_{i=0}^{\infty} \sigma(\Delta_i)}$. Define $B_z : L^{2}(F_{\infty},\mF{\infty}) \rightarrow Dom(\Delta_{\infty})$ by 
\begin{equation*}
	B_z = s-\lim_{i \rightarrow \infty} \Phi_i^{*}(\Delta_i -z)^{-1} \Pi_i.
\end{equation*}
Notice that $B_z$ is linear since all of its components are. Also each are bounded operators as well. For our choice of $z$ the distance from $z$ to $\overline{\bigcup_{i=0}^{\infty} \sigma(\Delta_i)}$ is positive so $\Phi_i^{*}(\Delta_i -z)^{-1} \Pi_i$ are bounded linear operators with norm bounded uniformly in $i$, so their limit is also bounded. That is, $B_z$ is a bounded linear operator on $L^{2}(F_{\infty},\mF{\infty})$. We claim that $B_z$ is the inverse of $(\Delta_{\infty} - z)$ contradicting the assumption that $z \in \sigma(\Delta_{\infty})$. Recall that $\Pi_i$ is a bounded linear and hence continuous operator. Let $f \in \bigcup_{i=0}^{\infty}\Phi_i^{*}Dom(\Delta_{i})$, then we have the following point-wise limit statement on a dense subspace of the domain of $\Delta_{\infty}-z$:
\begin{eqnarray}
	B_z(\Delta_{\infty} - z) f &=& \lim_{n \rightarrow \infty} \Phi_n^{*} (\Delta_n - z)^{-1} \Pi_n \lim_{m \rightarrow \infty} \Phi_m^{*} (\Delta_m - z ) \Pi_m f \notag\\
	&=&  \lim_{n \rightarrow \infty} \Phi_n^{*} (\Delta_n - z)^{-1} \Pi_n \Phi_M^{*} (\Delta_M - z ) \Pi_M f \label{eq:stabilize}\\
	&=&  \Phi_M^{*} (\Delta_M - z)^{-1} (\Delta_M - z ) \Pi_M f \notag\\
	&=& f.\notag
\end{eqnarray}
For large enough $m$ $ \lim_{m \rightarrow \infty} \Phi_m^{*} (\Delta_m - z ) \Pi_m f$ stabilizes to $\Phi_M^{*}(\Delta_M-z)\Pi_M f$ then as $n$ grows (\ref{eq:stabilize}) will also stabilize for $n \ge M$. Then since $\Delta_{\infty}$ is a closed operator the claim extends to $Dom(\Delta_{\infty})$. Finally by the decompositions in Lemmas \ref{lem:projL2} and \ref{lem:projF} the last limit equals $f$. Similar calculations can be used to show that $(\Delta_{\infty} - z)B_z = Id$. Thus there exists no $z \in \sigma(\Delta_{\infty})$ that is not in $\overline{\bigcup_{i=0}^{\infty} \sigma(\Delta_i)}$.
\end{proof}

In the standard theory of self-adjoint operators lie the spectral resolutions of self-adjoint operators \cite{Lax2002}. These spectral resolutions are orthogonal projection valued measures over $\B{R}$ supported on the spectrum of the operator they are representing. For $\Delta_{\infty}$ let $E_{\lambda}$ be the spectral resolution. Then
\begin{equation*}
	\Delta_{\infty} f = \int_{\sigma(\Delta_{\infty})} \lambda dE_{\lambda} f.
\end{equation*}
Note that for each $\lambda \in \mathbb{R}$, $E_{\lambda}:L^{2}(F_{\infty},\mF{\infty}) \rightarrow Dom(\Delta_{\infty})$ where for $f \not\in Dom(\Delta_{\infty})$ the integral fails to converge. We also have the orthogonal projections $\C{P}_i$ out of $Dom(\Delta_{\infty})$. 

From the previous discussion the following statement follows immediately.
 
\begin{theorem}\label{thm:specproj}
Let $E_{\lambda}$ be a spectral projection operator for $\Delta_{\infty}$. Then for all $\lambda \in \mathbb{R}$ and $i \in \mathbb{N}$
\begin{equation*}
	\F{D}'_i \cap E_{\lambda}(Dom(\Delta_{\infty})) = E_{\lambda}\F{D}'_i.
\end{equation*}
\end{theorem}

Similar statements could be made for $L^{2}(F_{\infty},\mF{\infty}),\ \mathcal{F}_{\infty}$, however we have not developed the notation for these spaces corresponding to the $\F{D}'_i$ notation. 

\begin{corollary}\label{cor:specdiscrete}
Suppose that $\mathcal{E}_0$ is a local regular Dirichlet form. 
%Take $G_i$ have uniformly bounded, finite cardinality. 
Assume that $\sigma(\Delta_i|_{\F{D}_i'}) \subset [M_i,\infty)$ where $\lim_{i \rightarrow \infty} M_i = \infty$ and $\sigma(\Delta_i)$ are all discrete. Then $\sigma(\Delta_{\infty}) = \bigcup_{i=0}^{\infty} \sigma(\Delta_n)$.
\end{corollary}

\begin{proof}
By Theorem \ref{thm:specdecomp}, 
\begin{displaymath}
	\sigma(\Delta_{\infty}) = \overline{\bigcup_{i=0}^{\infty} \sigma(\Delta_{\infty}|_{\F{D}'_i})}.
\end{displaymath}
However for any compact interval $[0,N]$
\begin{displaymath}
	\sigma(\Delta_{\infty}) \cap [0,N] = [0,N] \cap \bigcup_{i=0}^{M} \sigma(\Delta_{\infty}|_{\F{D}'_i})  
\end{displaymath}
for some $M = M(N)$ by hypothesis. Since each of the $\sigma(\Delta_i) = \bigcup_{j=0}^{i} \sigma(\Delta_{\infty}|_{\F{D}_j'})$ are discrete so is $\sigma(\Delta_{\infty}) \cap [0,N]$ for all $N$.
\end{proof}

The main point of this Corollary is that if the operators $\Delta_i$ have spectral gaps going to infinity then the closure in Theorem \ref{thm:specdecomp} adds no new points to the spectrum. There are many sufficient conditions for the two main hypotheses in Corollary \ref{cor:specdiscrete}. For example when computing the spectrum of $\Delta_{\infty}$ for Laakso spaces in Subsection \ref{ssec:Laakso} the spectrum of $\Delta_i$ can be computed directly and explicitly so that these hypotheses are straight forward to check. Also a metric measure space on which the Faber-Krahn inequality \cite{G} holds will satisfy the spectral gap hypothesis. Also if the resolvents of $\Delta_i$ are all known to be compact the spectral gap hypothesis will hold. 

\section{Examples}\label{sec:examples}

The two main classes of example considered here are the Laakso spaces where the horizontal space $F_0$ is taken to be the unit interval and the Sierpinski P\^ate \`a Choux where $F_0$ is a standard Sierpinski gasket. The P\^ate \`a Choux is a new construction suggested by Jean Bellisard. 

\subsection{The Laakso fractal}\label{ssec:Laakso}

\begin{figure}[t]
\begin{center}
\includegraphics[scale=.7]{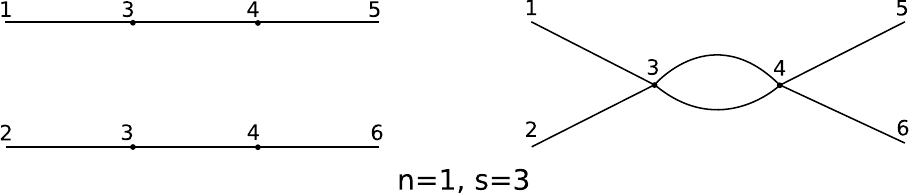}
\caption{The first step in constructing a Laakso space with horizontal space $F_0 = [0,1]$, and vertical set $G_1 = \{0,1\}$ with identifications made on $B_1 = \{\frac13, \frac23\}$.}
\label{fig:Laakso}
\end{center}
\end{figure}

Laakso spaces were initially introduced in \cite{Laakso2000} as the Cartesian product of a unit interval and a number of Cantor sets modulo an equivalence realtion. In \cite{Steinhurst2010,RS09} it was shown that they could also be constructed using the projective limit construction presented originally in \cite{BE04} and reiterated above. Take $F_0 = [0,1]$, the unit interval. Let $G_i = G = \{0,1\}$. Choose a sequence $\{j_l\}_{l=1}^{\infty}$ where $j_l \in \{j,j+1\}$ for some fixed integer, $j$, greater than one. Define
\begin{equation*}
	d_N = \prod_{j=1}^{N} j_i \hspace{1in} L_N = \left\{ \frac{i}{d_N} \right\}_{i=1}^{d_N-1}.
\end{equation*}
Then set $B_n = \phi_{n,0}^{-1} (L_n \setminus L_{n-1} )$. We have abbreviated $\phi_0 \circ \phi_1 \circ \cdots \circ \phi_n$ as $\phi_{n,0}$. The sets $L_N$ describe the location of what the quotient maps $\pi_i$ collapse and $d_N^{-1}$ the separation between the new identifications from any of the old identifications. A Laakso space will be denoted by $L$.

If $\C{E}_0$ is taken to be the standard Dirichlet form on the unit interval, namely $\C{E}_0(u,v) = \int_0^{1} \frac{du}{dx}\frac{dv}{dx} dx$ with the Sobolev space $H^{1,2}([0,1])$ as $\C{F}_0$, then there is a limiting Dirichlet form, $\C{E}_{\infty}$, on $L$ which has a generator $\Delta_{\infty}$. The analysis of the spectrum of $\Delta_{\infty}$ is the topic of \cite{RS09} and several chapters in \cite{Steinhurst2010}. Using the arguments involved in the proofs of Theorem \ref{thm:specdecomp} and Corollary \ref{cor:specdiscrete} the following explicit results hold.

\begin{theorem}[\cite{RS09}]\label{thm:laaksospec}
Let $L$ be a Laakso space with sequence $\{j_i\}$. The spectrum of $\Delta_{\infty}$ on this Laakso space is 
$$\sigma(\Delta_{\infty}) = \bigcup_{n=0}^{\infty} \bigcup_{k=1}^{\infty} \left\{ k^{2}\pi^{2}d_n^{2} \right\} 
\cup
 \bigcup_{n=2}^{\infty} \bigcup_{k=1}^{\infty} \left\{ k^{2}\pi^{2}4d_n^{2} \right\} 
% Internal pieces of the graph with Dirichlet boundary conditions.
\cup 
\bigcup_{n=1}^{\infty} \bigcup_{k=0}^{\infty} \left\{ (2k+1)^{2}\pi^{2}4d_n^{2} \right\}.$$ 
% Boundary intervals with Neumann - Dirichlet boundary conditions.
\end{theorem}

Existence of $\Delta_{\infty}$ follows from Theorem \ref{thm:InftyRegularity}. Theorem \ref{thm:specdecomp} reduces the calculation of $\sigma(\Delta_{\infty})$ to a calculation of $\sigma(\Delta_{\infty}|_{\F{D}_n'})$. Since $F_n$ is a quantum graph composed of intervals all of length $d_i$ with a very regular geometry, $\sigma(\Delta_{\infty}|_{\F{D}_n'})$ can be computed directly using counting arguments \cite{RS09}. Then the hypotheses of Corollary \ref{cor:specdiscrete} the union over $n$ is closed and is the entire spectrum of $\Delta_{\infty}$.

In fact, more in known including the multiplicities of the eigenvalues. Having the multiplicities allows computations of the spectral zeta function to be made and the analysis of physics-inspired problems possible \cite{Steinhurst2010,ST}. 

\subsection{Sierpinski P\^ate \`a Choux}\label{ssec:SP} 
The name of this example was suggested by Jean Bellisard who commented that such a space would evoke the memory of puff pastry in the reader. Denote by $SG$ the standard Sierpinski gasket constructed as the limit of the iterated function system $T_l(x) = \frac{1}{2}(x-q_l) + q_l$ for $l=0,1,2$ where $q_0 = (0,0)$, $q_1=(1,0)$, and $q_2=(\frac{1}{2}, \sqrt{3}/2)$. Define $V_0 = \{ q_0,q_1,q_2\}$ and $V_i = \{T_l{V_{l-1}}\}_{l=0,1,2}$ so $V_i$ is the set of vertices in the $i^{th}$ graph approximation to the Sierpinski gasket. Let $F_0 = SG$, $G_i = G = \{0,1\}$ and $B_i = \phi_{i-1,0}^{-1}(V_i \setminus V_{i-1})$. 

\begin{lemma}
The limit space $F_{\infty}$ is an infinitely ramified fractal with Hausdorff dimension $d_h = 1 + d_H(SG) = \frac{log(6)}{\log(2)}$ with respect to the geodesic metric. 
\end{lemma}

\begin{proof}
The cell structure on $F_{\infty}$ induced by the cell structures on $SG$ and on the Cantor set have boundaries that are themselves Cantor sets. Hence $F_{\infty}$ is infinitely ramified. Since $(\pi^{*}_{\infty,0})^{-1}(F_{\infty}) \subset SG \times \{0,1\}^{\B{N}} = SG \times K$, where $K$ is the Cantor set with contraction ration one half, the Hausdorff dimension is at most $\frac{log(6)}{\log(2)}$. This view of $F_{\infty}$ being ``unpacked'' into $F_0 \times G_1 \times G_2 \times \cdots$ was used in the proof of Theorem \ref{thm:InftyRegularity}.  By the same argument as in \cite{Steinhurst2013} it is at least $\frac{log(6)}{\log(2)}$.
\end{proof}

In light of Corollary \ref{cor:specdiscrete} it would be possible to write out explicitly the spectrum on $F_{\infty}$ as we did with the Laakso spaces. In particular, it is possible but somewhat involved to write the spectrum in a closed form. The reader can find solution to a  similar problem in \cite{RS-AT}. We note that, in the limit, the  Sierpinski P\^ate \`a Choux is not a Sierpinski fractafold, but the approximations $F_i$ are fractafolds. A fractafold, as briefly mentioned in the Introduction is a manifold where the local charts are maps into a reference fractal such as the Sierpinski gasget instead of into a Euclidean space. However, despite the fact that these fractafolds are very complicated, the spectrum of the Laplacian on $F_i$ can be found inductively using  methods presented in this paper. In particular,  the spectrum of each Laplacian $\Delta_i$ is a union of the spectrum of a large collection of  disjoint fractafolds (with Dirichlet boundary conditions). These fractafolds are rescaled copies of two kinds of finite fractafolds, and therefore the spectrum can be found using the methods of \cite{S}, and the standard rescaling by $5^n$. This is very similar to how the spectrum is found in the case of the Laakso spaces, described above. Finally, we can comment that the Laakso spaces are built using  intervals, which are one-dimensional analogs of the Sierpinski gasket. Therefore, in a sense, the Sierpinski P\^ate \`a Choux is a direct generalization of the Laakso spaces.    Combining the approaches of \cite{KKPSS,Steinhurst2010,RS09,S,RS-AT} one can study  all the eigenfunctions and eigenprojections, which will be subject of subsequent work.

\subsection{Connected fractal spaces isospectral to the fractal strings of Lapidus and  van Frankenhuijsen}\label{ssec:LvF}
Fractal strings are given a comprehensive treatment in \cite{LvF}, in particular in relation to spectral zeta functions, and we will only give a brief description here. We show that our construction can yield connected fractal spaces with Laplacians isospectral to the standard Laplacians on fractal strings. This implies, in particular, that there are symmetric irreducible diffusion processes whose generators are Laplacians with prescribed spectrum, as in the theory of fractal strings developed in \cite{LvF}.

A fractal string  is an open subset of $\mathbb R$, usually assumed to be a bounded subset, or at least that the lengths of the connected components are bounded and tend to zero.  Therefore it is a disjoint union of countably many finite intervals of lengths $l_i$. We will suppose that the intervals are indexed so that the lengths form a non-increasing sequence. By reindexing the fractal string with $l_i$ and $m_i$, unique lengths and multiplicities we can assume that $l_i$ is strictly decreasing. The Laplacian that we consider on $I$ is the usual Laplacian on an interval with Dirichlet boundary conditions on all the intervals. The eigenvalues of this Laplacian are all of the form 
\begin{equation*}
	\lambda_{i,k} = \frac{\pi^{2}k^{2}}{l_i^{2}}
\end{equation*}
with multiplicity $m_i$. What choices of $F_i$, $B_i$, and $G_i$ can be made to create a connected fractal with the same spectrum as a given fractal string? As the desire is to ``stitch'' the disjoint intervals together there is no expectation for a unique canonical method. 

Declare $F_{0} = [0,l_1]$ to be equipped with Dirichlet form $(\C{E}_0,\C{F}_0)$ where $\C{F}_0 = H^{1,2}([0,l_1])$ and $\C{E}_0(u,v) = \int_0^{l_1} u'v'$. Let  $B_1 = \{0,1\}$ and $G_1 = \{1,2,\ldots , m_1\}$. Then $F_1$ will be $m_1$ copies of the unit interval with left end points identified and right end points identified.  A particular implication  of this step is that $F_0=F_1$ if and  only if $m_1=1$. We impose zero boundary conditions at the endpoints, and therefore the spectrum of the Laplacian on $F_1$ is the spectrum on $F_0=[0,l_1]$ repeated, in the sense of multiplicity, $m_1$ times. For the next step     $G_2 = \{1, 2, \ldots , m_2+1\}$, and we choose 
\begin{align*}
	B_2 = \pi_1\left( ([0,l_1-l_2] \cup \{l_1\} ) \times G_1 \right) \cup \pi_1 \left( [l_1-l_2,l_1] \times ( G_1 \setminus \{1\}) \right)
\end{align*}
This implies that  the spectrum on $F_2$ is the union of the spectrum on $F_1$ and the spectrum on $ [l_1-l_2,l_1] \sim [0,l_2]$ repeated, in the sense of multiplicity, $m_2$ times. For $i \ge 1$ we take
\begin{align*}
	B_i =& \pi_{i,1}\left( ([l_1-l_i,l_1] \cup \{l_1\} ) \times G_i \right)\\
	& \cup \pi_{i,1}\left( [l_1-l_i,l_1] \times (G_1 \times \cdots \times G_i \setminus \{1, \ldots , 1\} \right).
\end{align*}
Where $G_j = \{1,\ldots, m_j+1\}$ for all $j \ge 2$. Recall the definition of $\pi_{i,1}$ from the proof of Theorem \ref{thm:InftyRegularity}. This construction is in a sense a non-self-similar version of the nested fractal construction. It is also somewhat similar to construction of some of the so called diamond fractals, see \cite{ADT,NT}.

In this setting Corollary \ref{cor:specdiscrete} holds since $\Delta_{\infty}|_{\F{D}_i'}$ consists of the eigenvalues for eigenfunctions present on $F_i$ but not on $F_{i-1}$. By construction these new eigenvalues are precisely the spectrum of the standard Laplacian on an interval of length $l_i$ with multiplicity $m_i$. The conclusion drawn from Corollary \ref{cor:specdiscrete} is that our construction does not introduce any new elements to the spectrum so the original fractal string and $F_{\infty}$ are isospectral. 

\section*{Acknowledgments}
The authors thank Eric Akkermans, Gerald Dunne, Patricia Alonso-Ruiz, Michel Lapidus and Jean Bellisard for many useful conversations and Jean Bellisard especially for the name and motivation for Sierpinski P\^ate \`a Choux example.
%An anonymous referee was especially generous in the time and care shown in his or her comments. 
%

\def\cprime{$'$}

\end{document}